\documentclass{article}
\usepackage{graphicx} % Required for inserting images
\usepackage{amsthm}
\usepackage{amsfonts}
\usepackage{amssymb}
\usepackage{amsmath}
\usepackage{mathtools}
\usepackage[maxbibnames=99]{biblatex}
\usepackage{xcolor}
\usepackage{esint} %needed for averaged integrals
\usepackage{todonotes}
\usepackage[citecolor=blue,colorlinks]{hyperref}
\usepackage[inner=2.6cm,outer=2.6cm,top=2.8cm,bottom=2.8cm]{geometry}

\mathtoolsset{showonlyrefs}

\makeatletter
\newcommand{\myitem}[1]{%
\item[#1]\protected@edef\@currentlabel{#1}%
}
\makeatother

\newtheorem{innercustomgeneric}{\customgenericname}
\providecommand{\customgenericname}{}
\newcommand{\newcustomtheorem}[2]{%
  \newenvironment{#1}[1]
  {%
   \renewcommand\customgenericname{#2}%
   \renewcommand\theinnercustomgeneric{##1}%
   \innercustomgeneric
  }
  {\endinnercustomgeneric}
}

\newcustomtheorem{customthm}{Theorem}
\newcustomtheorem{customlemma}{Lemma}

\newtheorem{thm}{Theorem}[section]

\newtheorem{lemma}[thm]{Lemma}
\newtheorem{proposition}[thm]{Proposition}

\newtheorem*{thm*}{Theorem}
\newtheorem*{corollary*}{Corollary}
\newtheorem*{lemma*}{Lemma}
\newtheorem*{proposition*}{Proposition}

\theoremstyle{definition}
\newtheorem{definition}[thm]{Definition}
\newtheorem*{definition*}{Definition}
\newtheorem{remark}[thm]{Remark}

\newtheorem*{remark*}{Remark}

\newcommand{\bb}[1]{\mathbb{#1}}
\newcommand{\ssf}[1]{\mathsf{#1}}
\newcommand{\supp}{\mathsf{supp}}

\newcommand{\lip}{\mathsf{lip}}
\newcommand{\Sec}{\mathsf{Sec}}
\newcommand{\Lip}{\mathsf{Lip}}

\newcommand{\Ric}{\mathsf{Ric}}

\newcommand{\RCD}{\mathsf{RCD}}

\newcommand{\CD}{\mathsf{CD}}
\newcommand{\aH}{\mathsf{H}}
\newcommand{\BV}{\mathsf{BV}}
\newcommand{\W}{\mathsf{W}}
\newcommand{\X}{X}

\newcommand{\sd}{\mathsf{d}}
\newcommand{\sL}{\mathsf{L}}

\newcommand{\m}{\mathfrak{m}}

\newcommand{\ssubset}{\subset\subset}

\newcommand{\mres}{\mathbin{\vrule height 1.6ex depth 0pt width
0.13ex\vrule height 0.13ex depth 0pt width 1.3ex}}

\makeatletter
\newcommand{\customlabel}[2]{%
   \protected@write \@auxout {}{\string \newlabel {#1}{{#2}{\thepage}{#2}{#1}{}} }%
   \hypertarget{#1}{#2}
}

\title{On perimeter minimizing sets in manifolds with quadratic volume growth}
\author{Alessandro Cucinotta\footnote{Mathematical Institute, University of Oxford \textit{E-mail}:  \href{mailto:alessandro.cucinotta@maths.ox.ac.uk}{alessandro.cucinotta@maths.ox.ac.uk}} \,and Mattia Magnabosco\footnote{Mathematical Institute, University of Oxford \textit{E-mail}:  \href{mailto:mattia.magnabosco@maths.ox.ac.uk}{mattia.magnabosco@maths.ox.ac.uk}}}
\date{}

\begin{document}

\maketitle

\begin{abstract}
    This paper studies whether the presence of a perimeter minimizing set in a Riemannian manifold $(M,g)$ forces an isometric splitting. We show that this is the case when $M$ has non-negative sectional curvature and quadratic volume growth at infinity. Moreover, we obtain that the boundary of the perimeter minimizing set is identified with a slice in the product structure of $M$.

    % In particular, we prove that,  of a perimeter minimizing set $E \subset M$ forces the isometric spliting $M \cong N \times \bb{R}$ for some other Riemannian manifold $N$. Moreover, with this identification, it holds $E \cong N \times \bb{R}_+$.
    
    % e and let $E \subset M$ be perimeter minimizing. Under which conditions on $M$, can we infer that $M \cong N \times \bb{R}$ and $E \cong N \times \bb{R}_+$, for some manifold $(N,g')$
   % Let $M$ be a Riemannian manifold with non-negative sectional curvature and quadratic volume at infinity. We show that, if there exists a perimeter minimizing set $E \subset M$, then $M$ splits isometrically as $N \times \bb{R}$ for some other Riemannian manifold $N$ and, with this identification, it holds $E \cong N \times \bb{R}_+$.
\end{abstract}

\section{Introduction}

A classical problem in calculus of variations is to determine the geometry of sets minimizing the perimeter in Euclidean space. A central result is that the only perimeter minimizing sets in $\bb{R}^n$ are Euclidean half spaces if and only if $n \leq 7$. Similarly, there exists non-affine solutions of the minimal surface equation on $\bb{R}^n$ if and only if $n \geq 8$.
A natural question is whether these results hold, in a generalized sense, in the setting of Riemannian manifolds. 

The rigidity properties of minimal graphs on Riemannian manifolds have been a recent topic of investigation. For example, assuming non-negative Ricci curvature, the only \emph{positive} solutions of the minimal surface equation are the constant functions by \cite{CMMR,Ding} (see also \cite{cmmr23}).
Without the positivity condition, solutions of the minimal surface equation on parabolic manifolds with non-negative Ricci curvature have vanishing Hessian by \cite{Colombo_2024}.

On the other hand, when considering general perimeter minimizing sets in Riemannian manifolds, a lot of properties can be obtained as a byproduct of the several known results on \emph{stable} minimal hypersurfaces. For instance, stable two-sided minimal hypersurfaces in Riemannian $3$-manifolds with non-negative Ricci curvature are totally geodesic by \cite{SchoenYau}, while other celebrated results along these lines were proved in \cite{FischerSchoen} and \cite{CurvYau}. More recently, it was shown in
\cite{chodosh2024completestableminimalhypersurfaces} that, in a $4$-manifold with non-negative sectional curvature, scalar curvature $\geq 1$, and weakly bounded geometry, every two-sided stable minimal hypersurface is totally geodesic (see also \cite{espinar2024frankelpropertymaximumprinciple} for a related result).

Nevertheless, when working with perimeter minimizing sets rather than stable minimal hypersurfaces, stronger rigidity results are to be expected.
This leads to the following question.

\medskip
\noindent\textbf{Question 1.} \label{Question1} Let $(M,g)$ be a Riemannian manifold and let $E \subset M$ be perimeter minimizing. Under which conditions on $M$, can we infer that $M \cong N \times \bb{R}$ and $E \cong N \times \bb{R}_+$, for some manifold $(N,g')$?

\medskip
% The available results on Question \ref{Question1} require strong assumptions on $M$. 
In \cite{Anderson}, it is shown that if $M$ has at most cubic volume growth, non-negative Ricci curvature and sectional curvature bounded from above (so that one also has a uniform lower sectional curvature bound), then the presence of a perimeter minimizer forces the universal cover of the manifold to split-off a real line (see also \cite{AndersonRodriguez} and \cite{AndersonAreaMin} for related results).

In \cite{EffectiveCarlotto}, it is shown that the only asymptotically flat $3$-manifold with non-negative scalar curvature which contains a perimeter minimizer is $\bb{R}^3$ (see also \cite{ChodoshScalar} for a related result). 
From \cite[Theorem 2]{CucinottaFiorani}, combining with the results from \cite{CheegerNaberCodim4}, it follows that the only Ricci-flat $4$-manifold with maximal volume growth containing a perimeter minimizing set is $\bb{R}^4$.

In view of the many rigidity results for manifolds with non-negative sectional or Ricci curvature, one expects to answer Question \ref{Question1} requiring only a lower curvature bound. In this paper, we present a result in this direction.

\begin{thm} \label{TMain}
    Let $(M^n,g,p)$ be a pointed Riemannian manifold with $\Sec_M \geq 0$ and such that
    \begin{equation} \label{Evol}
    \liminf_{r \to + \infty} \frac{\ssf{Vol}(B_r(p))}{r^2} < + \infty.
    \end{equation}
    If $E \subset M$ is perimeter minimizing, then $M \cong N \times \bb{R}$ and $E \cong N \times [0,+\infty)$.
\end{thm}

We remark that, a posteriori, the manifold $M$ from Theorem \ref{TMain} satisfies
\[
\limsup_{r \to + \infty} \frac{\ssf{Vol}(B_r(p))}{r^2} < + \infty.
\]
By \cite[Remark 3.11]{JostExistence} (see also \cite{ConiMorgan}), there exists a $4$-manifold with strictly positive sectional curvature which contains a perimeter minimizing set, so that Theorem \ref{TMain} fails if one asks for the volume growth at infinity to be at most quartic, instead of quadratic as in \eqref{Evol}. On the other hand, Theorem \ref{T1} below suggests that Theorem \ref{TMain} could hold even if the volume growth at infinity is at most cubic.

\medskip
We remark that, due to \cite{PetruninScalar}, $4$-manifolds with non-negative sectional curvature and scalar curvature $\geq 1$, have at most quadratic volume growth at infinity. 
Under these curvature assumptions (and also assuming weakly bounded geometry), stable two-sided minimal hypersurfaces are totally geodesic due to \cite{chodosh2024completestableminimalhypersurfaces}. Nevertheless, assuming only the stability of the hypersurface, no isometric splitting of the ambient space can be expected.
A consequence of Theorem \ref{TMain} is that, in $4$-manifolds with non-negative sectional curvature and scalar curvature $\geq 1$, replacing the local condition of stability with the global condition of being an area minimizing boundary,
one also obtains the global isometric splitting of the ambient space.
% The proof of Theorem \ref{TMain} relies on recently developed tools for the treatment of curvature lower bounds in synthetic sense. 
\\
\\
We now explain the main ideas behind Theorem \ref{TMain}. To this aim, we first briefly recall the proof that if $E \subset \bb{R}^n$ is perimeter minimizing and $n \leq 7$, then $E$ is a half space. By the monotonicity formula for minimal sets, the tangent cone at infinity $E_\infty \subset \bb{R}^n$ of $E$ is a perimeter minimizing cone. Since $n \leq 7$, the second variation formula for minimal cones forces $E_\infty$ to be a half-space. The rigidity case of the monotonicity formula then implies that the initial $E \subset \bb{R}^n$ is a half-space as well.

To prove Theorem \ref{TMain}, we repeat a similar argument in the setting of Riemannian manifolds. Unlike Euclidean spaces, Riemannian manifolds are not invariant under rescalings of the Riemannian metric. Hence, to repeat the aforementioned strategy, it becomes necessary to work in a larger class of spaces.
The right setting turns out to be the one of metric measure spaces with non-negative Ricci curvature and finite dimension in synthetic sense, i.e.\ $\RCD(0,N)$ spaces (see Section \ref{S1}).
% We recall that an $\RCD(K,N)$ space is a metric measure space where $N \in [1,+\infty)$ plays the role of an upper bound on the dimension, while $K \in \bb{R}$ plays the role of a lower bound on the Ricci curvature; this class includes measured Gromov-Hausdorff limits of smooth manifolds with uniform Ricci curvature lower bounds and finite dimensional Alexandrov spaces with sectional curvature bounded from below (see \cite{Asurv} for an account of the theory).
We stress that, even though the statement of Theorem \ref{TMain} only deals with sectional curvature lower bounds, the setting of Alexandrov spaces with non-negative sectional curvature would not be general enough to implement the aforementioned strategy.

We take an appropriate sequence of scales $r_i \uparrow + \infty$, and we consider a pointed measured Gromov-Hausdorff limit $(X,\sd,\m,p)$ of the spaces $(M,g/r_i^2)$ equipped with their renormalized volume measures. The metric space $(X,\sd)$ is a metric cone with non-negative sectional curvature, while $(X,\sd,\m)$ is an $\RCD(0,N)$ space. Moreover, there exists a set $E_\infty \subset X$ minimizing the perimeter (w.r.t. the metric measure structure on $X$). As in the Euclidean case, to conclude, it is sufficient to show that one has the isometric splitting $X \cong Y \times \bb{R}$ for some metric measure space $(Y,\sd_y,\m_y)$, and that, with this identification, it holds $E_\infty \cong Y \times \bb{R}_+$.

By condition \eqref{Evol} and the fact that $M$ contains a perimeter minimizer, $(X,\sd)$ has Hausdorff dimension at most $2$.  
If the Hausdorff dimension of $X$ is equal to $1$, the desired isometric splitting follows by standard arguments. Hence, we only study the case when $(X,\sd)$ has Hausdorff dimension exactly $2$. 
\par 
If $X$ is a cone over $S^1_R$ for some $R\in (0,1]$, i.e.\ $X=C(S^1_R)$, relying on the Splitting Theorem for $\RCD(0,N)$ spaces, we show that the only measure $\m$ so that $(C(S^1_R),\sd,\m)$ is $\RCD(0,N)$ is (a rescaling of) the two-dimensional Hausdorff measure. It then follows that $R=1$ and that $E_\infty \subset C(S^1_R) \cong \bb{R}^2$ is a half space, as claimed.

The non-trivial case is when $X$ is a cone over an interval, i.e.\ $X=C([0,l])$ for $l\in(0,\pi]$. We remark that there exists a measure $\m'$ on $C([0,l])$ such that $(C([0,l]),\sd,\m')$ is $\RCD(0,4)$, it contains a perimeter minimizer, and $l<\pi$. In particular, to treat this case, one cannot just rely on the fact that $(C([0,l]),\sd,\m)$ is an $\RCD(0,N)$ space containing a perimeter minimizer.
The key observation is that condition \eqref{Evol}, paired with the fact that $M$ contains a perimeter minimizer, implies that the volume of balls in $M$ grows at a uniform rate at infinity. 
This, combined with the concavity properties of $\RCD(0,N)$ densities on half lines, allows to deduce additional regularity for the limiting measure $\m$ on $X$. By a comparison argument (which relies on the recent results from \cite{Weak,lapb}), we then construct another measure $\tilde{\m}$ on $C([0,l])$ so that $(C([0,l],\sd,\tilde{\m})$ is $\RCD(0,N+1)$, the set $E_\infty \subset X$ is perimeter minimizing with respect to $\tilde{\m}$, and $\tilde{\m}$ converges to the $2$-dimensional Hausdorff measure at infinity. By taking another blow-down, we then deduce that $E_\infty$ minimizes the perimeter in $C([0,l])$ w.r.t. the $2$-dimensional Hausdorff measure, so that $l=\pi$, as claimed.

\medskip
We conclude by remarking that Theorem \ref{T1} below suggests that the optimal way to answer Question \ref{Question1} would be to require $\Ric_M \geq 0$ and 
\[
\int_1^{+\infty} \frac{t^2}{\ssf{Vol}(B_t(p))} \, dt=+\infty.
\]
However, assuming only non-negative Ricci curvature, very little is known on the structure of tangent cones at infinity (see the counterexamples in \cite{Colding1}). In particular, such tangent cones might not be unique and they might not be metric cones. Moreover, even if a tangent cone at infinity splits a line, the initial ambient space might fail to do so. Therefore, the blow-down procedure at the core of our argument does not easily adapt to manifolds with non-negative Ricci curvature. Finally, our strategy also crucially relies on the volume growth assumption \eqref{Evol}. Indeed, we apply the strong available results on manifolds with linear volume growth to the area minimizing boundary $\partial E$.

\subsection*{Acknowledgments}  
The authors wish to thank Daniele Semola and Andrea Mondino for inspiring discussions and suggestions.

 M.\,M.\;acknowledges support from the Royal Society through the Newton International Fellowship (award number: NIF$\backslash$R1$\backslash$231659).\ Part of this research was carried out by the authors at the Hausdorff Institute of Mathematics in Bonn, during the trimester program  ``Metric Analysis''. The authors wish to express their appreciation to the institution for the stimulating atmosphere and the excellent working conditions. For the purpose of Open Access, the authors have applied a CC BY public copyright licence to any Author Accepted Manuscript (AAM) version arising from this submission.

\section{Preliminaries} \label{S1}

A metric measure space is a triple $(X,\sd,\m)$, where $(X,\ssf{d})$ is a separable complete metric space and $\m$ is a locally finite Borel measure on $X$. Given a measurable set $A \subset X$, we denote by $\sL^1(A,\m)$ and $\sL^1_{loc}(A,\m)$ respectively integrable functions and locally integrable functions on $A$. Given an open set $\Omega \subset X$, we denote by $\Lip(\Omega)$, $\Lip_{loc}(\Omega)$, and $\Lip_c(\Omega)$ respectively Lipschitz functions, locally Lipschitz functions, and compactly supported Lipschitz functions on $\Omega$. If $f \in \Lip_{loc}(\Omega)$ and $x \in \Omega$ we set
\begin{equation} 
\lip(f)(x):= \limsup_{y \to x} \frac{|f(x)-f(y)|}{\ssf{d}(x,y)}.
\end{equation}

We briefly recall some facts on Ricci limit spaces and $\RCD(K,N)$ spaces.
In the foundational papers \cite{ChCo1, Colding1,ChcoStructure2,ChCo3}, Cheeger and Colding studied the structure of Ricci limit spaces, i.e. metric measure spaces arising as limits of manifolds of fixed dimension with a uniform lower bound on the Ricci curvature. We refer to the book \cite{CheegerLectures} and the references therein for an introduction to the topic.

$\RCD(K,N)$ spaces are metric measure spaces where $K$ plays the role of a lower bound on the Ricci curvature and $N$ plays the role of an upper bound on the dimension. They were introduced in \cite{AGS2} (in the case when $N=+\infty$) and \cite{Giglimem} (in the case when $N<+\infty$) following the seminal papers \cite{S1,S2,V}.
The class of $\RCD(K,N)$ spaces contains Ricci limit spaces and finite dimensional Alexandrov spaces with curvature bounded from below.
For a complete introduction to the topic, we refer to the survey \cite{Asurv}. 
From now on, when considering $\RCD(K,N)$ spaces, we always assume $N < +\infty$. The following key result follows from \cite{S2}.

\begin{thm}
    $\RCD(K,N)$ spaces are uniformly locally doubling.
\end{thm}

The previous result and Gromov's precompactness Theorem imply that the class of $\RCD(K,N)$ spaces is precompact w.r.t. the pointed measured Gromov-Hausdorff convergence (abbreviated pmGH). For the relevant background on this notion of convergence, we refer to \cite{GMS13}.
We only recall that in the case of a sequence of uniformly locally doubling metric measure spaces $(X_i,\sd_i, \m_i,x_i)$, pmGH convergence to $(X,\sd,\m,p)$ can be equivalently characterized by asking for the existence of a proper
metric space $(Z, \ssf{d}_Z )$ such that all the metric spaces $(X_i
, \sd_i)$ are isometrically embedded
into $(Z, \sd_Z )$, $x_i \to x$ and $\m_i \to \m$ weakly in $Z$. In this case we say that the convergence is realized in the space $Z$.

Theorem \ref{stability} below follows combining Gromov's precompactness Theorem with the stability of the $\RCD$ condition under pmGH convergence \cite{GMS13} (after \cite{S1,S2, V, AGS2}). 

\begin{thm} \label{stability}
    The class of pointed $\RCD(K,N)$ spaces with normalized measures is sequentially compact with respect to pointed measured Gromov-Hausdorff convergence.
\end{thm}

Another key result in the theory of $\RCD(0,N)$ spaces is the Splitting Theorem. We recall that, on manifolds with non-negative Ricci curvature, this result was proved by Cheeger and Gromoll in \cite{SplittingCheegerGromoll}. The generalization to Ricci limit spaces is due to Cheeger and Colding with their Almost-Splitting Theorem \cite{Colding1}. On metric measure spaces, the result is due to Gigli \cite{gigli2013splittingtheoremnonsmoothcontext}. We highlight that Gigli's version of the Splitting Theorem also ensures the splitting of the measures, a key fact that we will use later on.

\begin{thm}
    Let $(X,\sd,\m)$ be an $\RCD(0,N)$ space which contains a line. Then, there exists an $\RCD(0,N-1)$ space $(Y,\sd_y,\m_y)$ such that $X=Y \times \bb{R}$ as metric measure spaces.
\end{thm}

We conclude this brief overview with the definition of \emph{tangent cone at infinity} (or \emph{blow-down}) of an $\RCD(0,N)$ space.

\begin{definition} \label{DBlowdown}
    Let $(X,\sd,\m,x)$ be a pointed $\RCD(0,N)$ space, and consider a sequence $r_i \uparrow + \infty$. By Theorem \ref{stability}, up to a subsequence, the spaces $(X,\ssf{d}/r_i,\m(B_{r_i}(x))^{-1}\m,x)$ converge in pmGH sense to a limiting $\RCD(0,N)$ space $(X_\infty,\sd_\infty,\m_\infty,x_\infty)$. Such $X_\infty$ is called a \emph{tangent cone at infinity} (or \emph{blow-down}) of $X$.
\end{definition}

We recall that the tangent cone at infinity may not be unique and may not be a cone (see \cite{PerelmanTangent,Colding1}). On the other hand, if $(X,\sd)$ is a finite dimensional Alexandrov space with non-negative sectional curvature, then its tangent cone at infinity (which, in this case, is just a metric space) is a metric cone and it is unique (see, for instance, \cite[Theorem 2.11]{AntonelliPozzettaAlexandrov} and the references therein).

We now recall some facts on sets of finite perimeter and perimeter minimizing sets in $\RCD(K,N)$ spaces.
Sets of finite perimeter in metric measure spaces were studied in \cite{SomeFineProp,Am02,Mir,,AmbrosioDiMarinoBV}, among others. This theory was then further developed in the setting of $\RCD(K,N)$ spaces in \cite{ABS19,BPSrec,SemolaCutAndPaste}. 

\begin{definition}[Sets of locally finite perimeter]
Let $(X,\sd,\m)$ be a metric measure space and let $E \subset X$ be a Borel set. Given an open set $A \subset X$, the perimeter of $E$ in $A$ is defined as
$$
P(E, A) := \inf \left\{ \liminf_{k \to \infty} \int_A \mathsf{lip} f_k\, d\m \ : \ f_k \in \Lip_\mathrm{loc}(A)
, \ f_k \to \chi_E\ \mbox{ in } \sL^1_\mathrm{loc}(A,\m)\right\}.
$$
The set $E \subset X$ is said to have locally finite perimeter if $P(E, B_r(x)) < + \infty$ for all $x \in X$ and $r > 0$.
\end{definition}

\begin{definition} (Convergence in $\sL^1_{loc}$ sense)
    Let $(\X_i,\sd_i,\m_i,x_i)$ be a sequence of $\RCD(K,N)$ spaces converging in pmGH sense to $(Y,\sd,\m,y)$. The Borel sets $E_i \subset \X_i$ of finite measure converge in $\sL^1$ sense to a set $E \subset Y$ of finite measure if $\m_i(E_i) \to \m(E)$ and $1_{E_i}\m_i \to 1_E \m$ weakly in duality w.r.t. continuous compactly supported functions in the space $(Z,\sd_Z)$ realizing the pmGH convergence. 
    
    The Borel sets $E_i \subset \X_i$ converge in $\sL^1_{loc}$ sense to a set $E \subset Y$ if $E_i \cap B_r(x_i) \to E \cap B_r(y)$ in $\sL^1$ sense for every $r>0$.
\end{definition}

The next two propositions follow from \cite[Proposition 3.3 and Proposition 3.6]{ABS19}.
\begin{proposition} \label{P|compactness}
    Let $(\X_i,\sd_i,\m_i,x_i)$ be a sequence of $\RCD(K,N)$ spaces converging in pmGH sense to $(Y,\sd,\m,y)$. Let $E_i \subset X_i$ be sets with uniformly bounded measures such that $E_i \subset B_r(x) \subset Z$, where $(Z,\sd_Z)$ is the space realizing the convergence. If
    \[
    \sup_{i \in \bb{N}} P(E_i,X_i)< 
    + \infty,
    \]
    then, there exists a (non relabeled) subsequence and a set of finite perimeter $E \subset X$ such that $E_i \to E$ in $\sL^1$.
\end{proposition}

\begin{proposition} \label{P|lsc variation aperti}
    Let $(\X_i,\sd_i,\m_i,x_i)$ be a sequence of $\RCD(K,N)$ spaces converging in pmGH sense to $(\X,\sd,\m,x)$. 
    If $E \subset X$, and $E_i \subset X_i$ is a sequence such that $E_i \to E$ in $\sL^1$,
    then for every open set $A \subset Z$, where $(Z,\sd_Z)$ is the metric space realising the convergence, we have
    \[
    P(E,A) \leq \liminf_{i \to + \infty} P(E_i,A).
    \]
\end{proposition}

We now consider sets minimizing the perimeter in $\RCD(K,N)$ spaces.
Structural properties of perimeter minimizing sets in $\RCD(K,N)$ spaces were studied in \cite{Weak}, while other properties were then investigated in \cite{FSM,C1,C2}.

\begin{definition}[Perimeter minimizing sets]
    Let $(X,\sd,\m)$ be an $\RCD(K,N)$ space. A set of locally finite perimeter $E \subset X$ is perimeter minimizing if, for every bounded open set $U \subset X$, and for every set $C \subset X$ with $C \Delta E \ssubset U$, it holds $P(E,U) \leq P(C,U)$.

    Analogously, the set $E$ is sub-minimizing if the previous condition holds for any $C \subset X$ with $C \Delta E \subset \subset U$ and $C \subset E$.
    Finally, the set $E$ is super-minimizing if the previous condition holds for any $C \subset X$ with $C \Delta E \subset \subset U$ and $C \supset E$.
\end{definition}

The proof of the next lemma can be found in \cite[Proposition 1.2]{DePhilippsMagistrale} in the Euclidean setting. The same argument works for metric measure spaces.

\begin{lemma} \label{L4}
    Let $(X,\sd,\m)$ be an $\RCD(K,N)$ space. Let $E \subset X$ be a set which is sub-minimizing and super-minimizing. Then, $E$ is perimeter minimizing.
    % \begin{proof}
    %     Tesi magistrale de philippis Proposizione 1.2
    % \end{proof}
\end{lemma}

The next theorem comes from \cite[Theorem $4.2$ and Lemma $5.1$]{Dens}. We state the result for $\RCD(0,N)$ spaces, although it holds in the more general setting of PI spaces.

\begin{thm} \label{T7}
    % Let $(X,d,m)$ have infinite diameter.
    Let $(\X,\sd,\m)$ be an $\RCD(0,N)$ space. There exist $C,\gamma_0>0$ depending only on $N$ such that the following hold. If $E \subset \X$ is a perimeter minimizing set, then, up to modifying $E$ on an $\m$-negligible set, for any $x \in \partial E$ and $r>0$, it holds
    \[
    \frac{\m(E \cap B_r(x))}{\m(B_r(x))}>\gamma_0, \quad \frac{\m(B_r(x) \setminus E)}{\m(B_r(x))}>\gamma_0
    \]
    and
    \begin{equation}\label{eq:perimeterestimates}
        \frac{\m(B_r(x))}{Cr} \leq P(E,B_r(x)) \leq \frac{C\m(B_r(x))}{r}.
    \end{equation}
\end{thm}
\noindent From the previous result one deduces that locally perimeter minimizing sets admit both a closed and an open representative, and these have the same boundary which in addition is $\m$-negligible. Whenever we consider the boundary of a locally perimeter minimizing set, we will implicitly be referring to the boundary of its closed (or open) representative.

The next proposition is taken from \cite[Theorem $2.43$]{Weak}.

\begin{proposition} \label{P26}
Let $(\X_i,\sd_i,\m_i,x
_i)$ be a sequence of $\RCD(K,N)$ spaces converging in pmGH sense to $(Y,\sd,\m,y)$.
    Let $E_i \subset \X_i$ be a sequence of perimeter minimizing sets converging in $\sL^1_{loc}$ sense to $E \subset Y$. Then, $E$ is perimeter minimizing and, in the metric space realizing the convergence, it holds $\partial E_i \to \partial F$ in Kuratowski sense.
\end{proposition}

We conclude this section by stating and proving two technical lemmas which will be used to prove our main result.

\begin{lemma} \label{L7}
    Let $U \subset \bb{R}^n$ be an open convex set. Let $f \in \sL^1_{loc}(\bar{U})$ be a function such that $(\bar U,\sd,fd\lambda^n)$ is an $\RCD(0,N)$ space. Then $f \in \Lip_{loc}(U)$.
    \begin{proof}
        Let $\nu \in S^{n-1}$ be fixed. By \cite{MondinoInventiones}, for $\aH^{n-1}$-a.e. line $l$ parallel to $\nu$, the restricted function $f_l:l \cap U \to \bb{R}_+$ is a $\ssf{CD}(0,n)$ density on $l \cap U$. In particular, for every such line $l$, it follows that $f_l$ has a locally Lipschitz representative %(cf. \cite[4.2]{MondinoInventiones}) 
        and that $f_l^{1/(n-1)}$ is concave.

        %We restrict to an open set $K \ssubset U$. We claim that $f \in \sL^{\infty}(K)$. 
        We claim that for every open set $K \ssubset U$ it holds that $f \in \sL^{\infty}(K)$. Let $x_0 \in K$ be a Lebesgue point of $f$. Let $i \in \bb{N}$ be such that there is a Lebesgue point $x_i \in K$ of $f$ such that $f(x_i) \geq i$. Consider $\nu_i:=(x_i-x_0)/|x_i-x_0|$ and consider the restrictions of $f$ to lines parallel to $\nu_i$. Given $\varepsilon>0$ small, since $x_0$ and $x_i$ are Lebesgue points, there exists $r>0$ such that
        \[
        \lambda^n(\{x \in B_r(x_0):f(x) \leq f(x_0)+1\}) \geq (1-\varepsilon)\lambda^n(B_r(x_0))
        \]
        and
        \[
        \lambda^n(\{x \in B_r(x_i):f(x) \geq i-1\}) \geq (1-\varepsilon)\lambda^n(B_r(x_i)).
        \]
        Hence, there exists a set $A$ of lines parallel to $\nu_i$ of strictly positive $\aH^{n-1}$ measure such that, for every $l \in A$, it holds
        \[
        \lambda^1(l \cap \{x \in B_\varepsilon(x_0):f(x) \leq f(x_0)+1\}) >0
        \]
        and
        \[
        \lambda^1(l \cap \{x \in B_\varepsilon(x_i):f(x) \geq i-1\}) >0.
        \]
        Let $l \in A$, since the Lipschitz representative of $f_l^{1/(n-1)}$ is positive, concave, and it attains a value $\leq f(x_0)+1$ and a value $\geq i-1$ on $K \cap l$, then $i \leq c(K,U,f)$.
        This proves that $f$ restricted to its Lebesgue points in $K$ is bounded above by a constant, so that $f \in \sL^\infty(K)$.

        We now prove that $f$ is Locally Lipschitz in $K$. Let $\nu \in S^{n-1}$ be fixed and let $l$ be any line parallel to $\nu$ such that $f_l^{1/(n-1)}$ is positive and concave in $l \cap U$ and bounded in $l \cap K$. Since $l \cap U \ssubset l\cap K$, the positivity and concavity of $f_l^{1/(n-1)}$ guarantee that there is a constant $c_{K}>0$ such that, if $|(f_l^{1/(n-1)})'|(x) \geq m$ for some $m>0$ and some $x \in l \cap K$, then $f_l^{1/(n-1)}(x) \geq c_Km$. Therefore, using that $f_l^{1/(n-1)}$ is bounded in $l \cap K$, we deduce that $(f_l^{1/(n-1))})'$ is itself bounded in $l \cap K$. 
    %    Since $f_l^{1/(n-1)}$ is positive and concave, and $l \cap U$ is a (possibily infinite) interval, there exists a constant $c_{K}>0$ such that if $|(f_l^{1/(n-1)})'|(x) \geq t$ for some $t>0$ and some $x \in K$, then $f_l^{1/(n-1)}(x) \geq c_Kt$. Combining with the fact that $f$ is bounded in $K$ by a constant $C(K,U,f)$, it follows that, possibly changing the constant, $|(f_l^{1/(n-1)})'|(x) \leq C(K,U,f)$ for some constant.
    By \cite[Theorem 4.21]{EvansGariepy}, it follows that $f_l^{1/(n-1)} \in \W^{1,\infty}_{loc}(K)$. Since $K \ssubset U$ was arbitrary, it holds $f_l^{1/(n-1)} \in \W^{1,\infty}_{loc}(U)$, concluding the proof.
    \end{proof}
\end{lemma}

\begin{lemma} \label{L6}
    Consider a cone $C([0,l])$ for some $0<l\leq \pi$, which, equipped with a measure $\m=f\aH^2$, with $f\in \Lip_{loc}\big(\mathrm{int}(C([0,l]))\big)$, is an $\RCD(0,N)$ space. Let $Y=C([a,b]) \subset Z$ for some $0 \leq a < b \leq l$ and let $p$ be the tip of $C([0,l])$. Denoting by $P_\m(\cdot, \cdot)$ perimeters in $C([0,l])$ w.r.t. the measure $\m$,  for every $s>0$ it holds that 
    \[
         P_{\m}(Y, B_s(p)) =
         \begin{cases}
             \displaystyle \int_0^s f_{|C(\{a\})}(z)+f_{|C(\{b\})}(z) \, dz  &\text{if }a\neq 0, b\neq l,\\
             \displaystyle\int_0^s f_{|C(\{a\})}(z) \, dz &\text{if }a\neq 0, b= l,\\
             \displaystyle\int_0^s f_{|C(\{b\})}(z) \, dz&\text{if }a= 0, b\neq l.
         \end{cases}
         \]
\end{lemma}

\begin{proof}
    We prove the case $a\neq 0, b= l$, the other cases can be done analogously.
    
    Given a point $q\in C(\{a\})\setminus\{p\}$, $f$ is Lipschitz and thus bounded in a neighborhood $B$ of $q$ in $C([0,l])$. Reasoning as in the proof of the previous lemma, for $\aH^{n-1}$-a.e.\ line $l$ parallel to $C(\{a\})$, the restricted function $f_l$ is such that $f_l^{1/(n-1)}$ is concave. By Lemma \ref{L7}, it follows that $f_l^{1/(n-1)}$ is concave \emph{for every} line $l$ parallel to $C(\{a\})$. Since $f$ is bounded in $B$, we conclude that $f$ is also bounded in a neighbourhood of $p$.
    
    Now, call $\sd_a$ the signed distance from $C(\{a\})$ and let $\phi:\bb R\to[0,1]$ be defined as
    \begin{equation}
        \phi(t)= \begin{cases}
            0 &\text{if }t\leq 0,\\
            t &\text{if }t\in [0,1],\\
            1 &\text{if }t\geq 1.\\
        \end{cases}
    \end{equation}
    Then, for every $n\in \bb N$, we consider $u_n\in \Lip_{loc}\big(C([0,l])\big)$ defined as $u_n(x)=\phi(n\sd_a(x))$. Considering $\pi(x):= \sd(p, \tilde \pi(x))$ where $\tilde \pi$ is the closest point projection on $C(\{a\})$ observe that 
    \begin{multline}
        \liminf_{n\to\infty} \int_{B_s(p)} |\nabla u_n| \, d \m = \liminf_{n\to\infty} \int_0^s \int_{\pi^{-1}(s)} n \cdot f 1_{B_s(p)\cap \{0\leq\sd_a\leq1/n\}} \, d\lambda^1 \, d\lambda^1= \int_0^s f_{|C(\{a\})}(z)\, d z,
    \end{multline}
    where the last step uses the dominated convergence theorem, thanks to $f$ being locally bounded around $p$. We deduce that $P_{\m}(Y, B_s(p)) \leq\int_0^s f_{|C(\{a\})}(z) \, dz$.

    Assume by contradiction that $P_{\m}(Y, B_s(p)) <\int_0^s f_{|C(\{a\})}(z) \, dz$, then there exist $s_1,s_2\in (0, s)$ with $s_1<s_2$ such that 
    \begin{equation}
        P_{\m}(Y, B_{s_2}(p)\setminus B_{s_1}(p)) <\int_{s_1}^{s_2} f_{|C(\{a\})}(z) \, dz. 
    \end{equation}
    Now, for every $\delta>0$, we can find a subinterval $I^\delta= [s^\delta_1,s^\delta_2]\subset [s_1,s_2]$ with $|I^\delta|<\delta$ such that 
    \begin{equation}\label{eq:Idelta}
           \int_{s^\delta_1}^{s^\delta_2} f_{|C(\{a\})}(z) \, dz - P_{\m}(Y, B_{s^\delta_2}(p)\setminus B_{s^\delta_1}(p)) > c |I^\delta|,
    \end{equation}
    for a positive constant $c>0$. This can be proved by taking finer and finer partitions of $[s_1,s_2]$ and selecting suitable subintervals.
    By uniform continuity of $f_{|C(\{a\})}$ on $[s_1,s_2]$, we take $\delta$ such that 
    \begin{equation}
         \int_J f_{|C(\{a\})}(z) \, dz - |J| \cdot \inf_{J} f_{|C(\{a\})} < \frac c2 |J|,
    \end{equation}
    on any interval $J\subset [s_1,s_2]$ with $|J|<\delta$. In particular, for the interval $I^\delta$ satisfying \eqref{eq:Idelta}, we obtain 
    \begin{equation}
        |I^\delta| \cdot \inf_{I^\delta} f_{|C(\{a\})} - P_{\m}(Y, B_{s^\delta_2}(p)\setminus B_{s^\delta_1}(p)) > \frac c2 |I^\delta|>0.
    \end{equation}
    We can then find an open neighborhood $A$ of $C(\{a\}) \cap (B_{s^\delta_2}(p)\setminus B_{s^\delta_1}(p))$ such that $P_{\m}(Y, A) < (s^\delta_2-s^\delta_1)\inf_A f$. 
    However, calling $\bar f=\inf_A f$, we have $P_{\m}(Y, A) \geq  P_{\bar f \lambda^2}(Y, A)= \bar f P_{\lambda^2}(Y, A) \geq \bar f (s^\delta_2-s^\delta_1)$, a contradiction.
\end{proof}

\section{Main result}

The next result shows that perimeter minimizing sets in manifolds with non-negative Ricci curvature, and sufficiently slow volume growth at infinity, are regular. This is an adaptation of \cite[Theorem 2.1]{Anderson}.
\begin{thm} \label{T1}
    Let $(M^n,g,p)$ be a pointed Riemannian manifold with $\Ric_M \geq 0$ and such that
    \begin{equation}\label{eq:weakassump}
        \int_1^\infty \frac{t^2}{\ssf{Vol}(B_t(p))} \, dt =+\infty.
    \end{equation}
    If $E \subset M$ is perimeter minimizing, then $E$ is smooth and its boundary is totally geodesic.
    \begin{remark}
       Theorem \ref{T1} proves that the area minimizing boundary $\partial E$ is totally geodesic. Similar statements for \emph{stable} minimal hypersurfaces are obtained in \cite{SchoenYau, FischerSchoen,CurvYau,chodosh2024completestableminimalhypersurfaces}. The main difference is that the stronger assumption that the minimal hyperurface is an area minimizing boundary allows to use the estimates from Theorem \ref{T7}. This is the reason why Theorem \ref{T1} has a simpler proof than the forementioned results. 
    \end{remark}
    \begin{proof}
        Let $\Sigma \subset \partial E$ be the singular set of $\partial E$. By the classical regularity theory for perimeter minimizers, $\Sigma$ is a closed set with $\aH^{n-7}(\Sigma)=0$. Moreover, by the stability inequality (see \cite{ColdingMinicozzi}), for every $\phi \in C^\infty_c(\partial E \setminus \Sigma)$, it holds
        \begin{equation} \label{E1}
        \int_{\partial E} \phi^2(|\Pi_{\partial E}|^2+\Ric(\nu,\nu)) \, d\aH^{n-1} \leq \int_{\partial E } |\nabla_{\partial E} \phi|^2 \, d\aH^{n-1},
        \end{equation}
        where $\nu$ is the normal to $\partial E$ and $\Pi_{\partial E}$ is the second fundamental form of $\partial E$ (both are only defined in the smooth points).
        By approximation, inequality \eqref{E1} holds for any function $\phi \in \Lip_c(\partial E \setminus \Sigma)$. We now divide the remaining part of the proof in three different steps. \medskip

        \textbf{Step 1}: Inequality \eqref{E1} holds for any function $\phi \in \Lip_c(M)$. \medskip

        \noindent To prove this, fix $\phi \in \Lip_c(M)$. It is sufficient to find a sequence of functions $\eta_i \in \Lip_c(M)$ taking values in $[0,1]$ such that $\eta_i \equiv 1$ on a neighbourhood of $\supp(\phi) \cap \Sigma$ and $\int_{\partial E } |\eta_i|^2 + |\nabla_{\partial E} \eta_i|^2 d\aH^{n-1}  \to 0$ as $i\to \infty$. Indeed, given such a sequence, it holds
        \begin{multline} 
        \int_{\partial E}  ((1-\eta_i)\phi)^2(|\Pi_{\partial E}|^2+\Ric(\nu,\nu)) \, d\aH^{n-1} 
        \leq \int_{\partial E } |\nabla_{\partial E} (1-\eta_i)\phi|^2 \, d\aH^{n-1} \\
         \leq 
        \int_{\partial E } (1-\eta_i)^2|\nabla_{\partial E} \phi|^2 \, d\aH^{n-1}+c_1(\phi)\int_{\partial E } |\eta_i|^2 + |\nabla_{\partial E} \eta_i|^2 d\aH^{n-1}.
\end{multline}        
Passing to the limit as $i\to \infty$ the last inequality, we would conclude the proof of Step 1.

        We now construct the functions $\eta_i \in \Lip_c(M)$ with the desired properties repeating an argument of \cite[Between equations (2.6) and (2.7)]{Anderson} (after \cite{SingActa}). 
        Let $M$ be isometrically embedded in a large Euclidean space $\bb{R}^L$. Let $\varepsilon>0$. Since $\aH^{n-7}(\Sigma)=0$, there exists a finite collection $\{Q_k\}_k$ of cubes in $\bb{R}^L$ with sides $s_k \leq \varepsilon$ such that
        \[
        \supp(\phi) \cap \Sigma \subset \bigcup_k Q_k, \quad \text{and} \quad
        \sum_k s_k^{n-7} \leq \varepsilon.
        \]
        By relabeling, we can suppose that $s_1 \geq s_2 \geq \cdots$. By \cite[Lemmas 3.1 and 3.2]{SingActa}, there exists a function $\eta_\varepsilon \in C^\infty_c(\bb{R}^L)$ taking values in $[0,1]$ such that $\eta_\varepsilon \equiv 1$ on $\cup_k Q_k$, $\supp(\eta_\varepsilon) \subset \cup_k (3/2)Q_k$, and
        \[
        |\nabla_{\bb{R}^L} \eta_\varepsilon| \leq c s_k^{-1} \quad \text{on } T_k:=(3/2)Q_k \setminus \cup_{j>k}(3/2)Q_j,
        \]
        for a constant $c>0$ depending only on $L$.
        Observe that if $\varepsilon>0$ is small enough, since $M$ is embedded isometrically in $\bb{R}^L$, it holds $B^M_{3 \sqrt{L} s_k}(x) \supset ((3/2)Q_k) \cap M$ for any $x \in Q_k \cap M \cap \supp(\phi)$. We can suppose that each cube $Q_k$ intersects $\Sigma \cap \supp(\phi)$, so that using Theorem \ref{T7}, it holds
        \[
        \aH^{n-1}(\partial E \cap (3/2)Q_k) \leq \aH^{n-1}(\partial E \cap B^M_{3 \sqrt{L}s_k}(x)) \leq c(n,L)s_k^{n-1}.
        \]
        Hence, using that $|\nabla_{\partial E}| \leq |\nabla_{\bb{R}^L}|$, for a constant $c'$ depending on $L$ and $n$, it holds
        \[
        \int_{\partial E} |\nabla_{\partial E} \eta_\varepsilon|^2 \, d\aH^{n-1} \leq c'\sum_k \aH^{n-1}(T_k \cap \partial E)s_k^{-2} \leq c' \sum_k s_k^{n-3} \leq c'\varepsilon.
        \]
        Similarly,
        \[
        \int_{\partial E} | \eta_\varepsilon|^2 \, d\aH^{n-1} \leq \sum_k \aH^{n-1}(T_k \cap \partial E) \leq c'\sum_k s_k^{n-1} \leq c'\varepsilon.
        \]
        Considering functions $\eta_{\varepsilon_i}$ for a sequence $\varepsilon_i \downarrow 0$, we conclude the proof of Step 1.
        
        \medskip
        \textbf{Step 2}: $\Pi_{\partial E} \equiv 0$ on the set $\partial E \setminus \Sigma$.\medskip
        
        \noindent Let $x \in \partial E$, $R>0$, and consider the function $\phi_R \in \Lip_c(M)$ defined by
        \[
        \phi_R(y):=\frac{\int_{1 \vee \sd(x,y) \wedge R}^R \frac{s}{P(E,\bar{B}_s(x))} \, ds}{\int_1^R\frac{s}{P(E,\bar{B}_s(x))} \, ds}.
        \]
        To shorten the notation, we set $C_R:=\int_1^R\frac{s}{P(E,\bar{B}_s(x))} \, ds$.
        It holds that
        \begin{align} \label{E5}
            \int_{\partial E} & |\nabla_{\partial E} \phi_R|^2 \, d\aH^{n-1} 
            \leq
            \int_{\partial E} |\nabla \phi_R|^2 \, d\aH^{n-1}
           C_R^{-2}\int_{(B_R(x) \setminus B_1(x)) \cap \partial E} \frac{\sd(x,y)^2}{P(E,\bar{B}_{\sd(x,y)}(x))^2} \, d\aH^{n-1}(y).
        \end{align}
        We define $h :[1,R] \to \bb{R}$ as $h(s):=P(E,\bar{B}_{s}(x))$. Observe that $h$ has bounded variation, since it is monotone. We also consider the measure $\nu$ on the Borel sets of $[1,R]$ defined as $\nu([a,b]):=P(E,\bar{B}_b(x) \setminus B_a(x))$ for every $1 \leq a \leq b \leq R$. The measure
        $\nu$ is the distributional derivative of $h$ in $[1,R]$. Moreover, we have that
        \[
        \nu=\sd(x,\cdot)_\#\big[\aH^{n-1} \mres \big(\partial E \cap (B_R(x) \setminus B_1(x))\big)\big],
        \]
        and therefore
        \begin{align} \label{E6}
            \int_{\partial E} &|\nabla_{\partial E} \phi_R|^2 \, d\aH^{n-1} 
            \leq   C_R^{-2}\int_1^R \frac{s^2}{h(s)^2} \, d\nu(s).
        \end{align}
        For a function of bounded variation $f \in \BV(\bb{R})$, we denote by $D^j f$ the jump part of its derivative, by $\tilde{D} f$ the remaining part of the derivative and by $J_f$ the jump set. Using the chain rule (see \cite[Theorem $3.96$]{AmbFuscPall}), it holds
        \begin{align*}
        D \Big(\frac{s^2}{h(s)} \Big)
        &=
        D^j \Big(\frac{s^2}{h(s)} \Big)+\tilde{D} \Big(\frac{s^2}{h(s)} \Big) \\
        &=s^2\Big( \frac{1}{h^+(s)}-\frac{1}{h^-(s)} \Big) \aH^0 \mres J_h+
        2\frac{s}{h(s)} d\lambda^1-\frac{s^2}{h(s)^2} \tilde{D}h \\
        & = s^2\Big( \frac{1}{h^+(s)}-\frac{1}{h^-(s)} \Big) \aH^0 \mres J_h+
        2\frac{s}{h(s)} d\lambda^1-\frac{s^2}{h(s)^2} \nu+\frac{s^2}{h(s)^2} D^jh.
        \end{align*}
        By definition of $h$, on a jump point of $h$, it holds $h(s)=h^+(s)$, so that 
        \begin{multline}
         s^2\Big( \frac{1}{h^+(s)}-\frac{1}{h^-(s)} \Big) \aH^0 \mres J_h+\frac{s^2}{h(s)^2} D^jh=s^2(h^+(s)-h^-(s))\Big(\frac{1}{h^+(s)^2}-\frac{1}{h^+(s)h^-(s)} \Big) \aH^0 \mres J_h \leq 0.
        \end{multline}
        Combining with the previous chain of inequalities, we obtain
        \[
        \frac{s^2}{h(s)^2} \nu \leq 2\frac{s}{h(s)} d\lambda^1-D \Big(\frac{s^2}{h(s)} \Big).
        \]
        In particular, we deduce that
        \[
        C_R^{-2}  \int_{1}^R \frac{s^2}{h(s)^2} \, d \nu (s)=2 C_R^{-1}-C_R^{-2}\Big( \frac{R^2}{h(R)}-\frac{1}{h(1)}\Big) \leq 
        2 C_R^{-1}+C_R^{-2}\frac{1}{h(1)}
        .
        \]
        By the hypothesis \eqref{eq:weakassump} combined with the perimeter estimate of Theorem \ref{T7}, we observe that $C_R \to +\infty$ as $R \to + \infty$. Combining the last inequality with \eqref{E6} and Step $1$, we deduce that $\Pi_{\partial E} \equiv 0$ in $B_1(x) \cap (\partial E \setminus \Sigma)$. By a scaling argument, it holds $\Pi_{\partial E} \equiv 0$ in $\partial E \setminus \Sigma$.\medskip

        \textbf{Step 3}: $\Sigma=\emptyset$.\medskip

        \noindent Assume by contradiction that this is not the case. We use Federer's dimension reduction argument to find a contradiction. Let $p \in \Sigma$. Taking the blow-up of $E$ in $p$, we obtain a perimeter minimizing set $E_1 \subset \bb{R}^n$. Since $E$ is singular in $p$, the origin $0 \in \bb{R}^n$ belongs to the singular set $\Sigma_1$ of $\partial E_1$.
        
        We claim that, since $\partial E$ is totally geodesic outside of its singular set, the same holds for $\partial E_1$. 
        To see this, let $M$ be isometrically embedded into a large Euclidean space $\bb{R}^L$. As we take the blow-up of $M$ at $p$, the second fundamental form of $M$ w.r.t. $\bb{R}^L$, in the Euclidean unit ball around $p$, converges uniformly to zero. Hence, in the same Euclidean ball, also the second fundamental form of $\partial E$ w.r.t. $\bb{R}^L$ converges uniformly to zero. By \cite[Theorem 5.3.2]{HutchinsonVarifolds} (see also \cite{MondinoVarifolds}), it follows that $\partial E_1 \setminus \Sigma_1$ is totally geodesic in $\bb{R}^L$, so that it is also totally geodesic in the copy of $\bb{R}^n$ inside $\bb{R}^L$ which corresponds to the tangent space to $M$ in $p$, proving the claim.
        Moreover,  taking another blow-up in the origin, we can additionally assume that $E_1$ is a cone. 

        If $\Sigma_1=\{0\}$, then $\partial E_1$ is totally geodesic outside of the origin. As a consequence, $\partial E_1$ is a hyperplane, which contradicts the fact that $0 \in \Sigma_1$. Hence, we can suppose that there exists $p \neq 0$ such that $p \in \Sigma_1$. Taking a blow-up of $E_1$ in $p$, and using that $E_1$ is a cone, we obtain a perimeter minimizing set $\tilde{E}_2 \subset \bb{R}^n$ of the form $\tilde{E}_2 = E_2 \times \bb{R} \subset \bb{R}^{n-1} \times \bb{R}$. Moreover, $\partial \tilde{E}_2$ is totally geodesic outside of its singular set, and $0$ belongs to the singular set $\tilde{\Sigma}_2$ of $\partial \tilde{E}_2$. Hence, $E_2 \subset \bb{R}^{n-1}$ is perimeter minimizing, $\partial E_2$ is totally geodesic outside of its singular set, and $0$ belongs to the singular set $\Sigma_2$ of $\partial E_2$. Taking another blow-up in the origin we can additionally assume that $E_2$ is a cone. As before, if $\Sigma_2 = \{0\}$, we obtain a contradiction. Otherwise, we keep repeating this procedure, until we obtain $E_k \subset \bb{R}^{n-k+1}$ as before and such that $\Sigma_k=\{0\}$. This happens for some $k \geq n-7$ by the standard regularity theory of perimeter minimizers. 
        %(questo Step si puo mettere come lemma a parte: Se un Per Min è totally geodesic fuori dal singular set, allora è regolare).
    \end{proof}
\end{thm}

% \begin{remark}
% From Step $3$ of the previous theorem, it follows that if a perimeter minimizing set in a Riemannian manifold is totally geodesic outside of its singular set, then it is regular.
% \end{remark}

\begin{remark}\label{R1}
We say that a closed set $A \subset M^n$ is smooth if, for every $x \in A$, there exists a chart $(U,\phi)$ of $M$ such that $\phi(A \cap U) \subset \bb{R}^n$ is either the whole $\bb{R}^n$ or a half space $\bb{R}^{n-1} \times \bb{R}_+$.
    We recall that if $E \subset M$ is a perimeter minimizing set whose (essential) boundary is smooth, then its closed representative is a smooth set.
\end{remark}

We now prove two simple lemmas involving the volume growth condition \eqref{eq:weakassump}.

\begin{lemma} \label{Lemma:connectedness}
    Let $(M^n,g,p)$ be a pointed Riemannian manifold with $\Sec_M \geq 0$ and such that
    \[
    \int_1^\infty \frac{t^2}{\ssf{Vol}(B_t(p))} \, dt =+\infty.
    \]
    If $E \subset M$ is perimeter minimizing, then $\partial E$ is connected. Moreover, the closed representative of $E$ is a connected smooth set.
    \begin{proof}
        Assume by contradiction that $\partial E$ is disconnected. Modulo replacing $E$ with its complement, there exists a connected component $A \subset E$ such that $\partial A$ has more than one connected component. By the previous theorem and Remark \ref{R1}, $A$ is a smooth set. By \cite[Theorem 5.2]{Burago_1977}, $A \cong N^{n-1} \times [0,l]$ with its intrinsic metric, for some manifold $N^{n-1}$ with non-negative sectional curvature. If $N$ is compact, then $E \setminus A$ is a competitor of $E$, contradicting that $E$ is perimeter minimizing. Hence, $N$ is non-compact. 
        
        Let $p \in N$. We claim that there exists $s>0$ such that
    \begin{equation} \label{E7}
    l P(B^N_s(p),N) < 2 \aH^{n-1}(B^N_s(p)).
    \end{equation}
    Recall that, by the coarea formula, the function $ s \mapsto \aH^{n-1}(B^N_s(p))$ is absolutely continuous and satisfies $\frac{d}{ds} \aH^{n-1}(B^N_s(p)) = P(B^N_s(p),N)$ for a.e. $s>0$.
    Hence, if \eqref{E7} fails for every $s>0$, it follows that $\frac{d}{ds} \aH^{n-1}(B^N_s(p)) \geq 2l^{-1} \aH^{n-1}(B^N_s(p))$ and thus $\aH^{n-1}(B^N_s(p))$ grows exponenentially in $s$. This contradicts the Bishop-Gromov inequality and proves the claimed inequality \eqref{E7}.
    
    So let $s_0>0$ be a value satisfying \eqref{E7}.
    Consider the set
    \[
    B:= E \setminus (B^N_{s_0}(p) \times [0,l]).
    \]
    Let $U \subset M$ be an open set containing the subset of $A \subset E$ identified with $\bar{B}^N_{s_0+1}(p) \times [0,l]$.
    Observe that $B \Delta E \ssubset U$. Moreover, it holds that
    \begin{align*}
    P (B,U)
    &=l P(B^N_{s_0}(p),N)+P(E,U \setminus \bar{B}^N_{s_0}(p) \times [0,l])
    \\
    &< 2 \aH^{n-1}(B^N_{s_0}(p))
    +P(E,U \setminus \bar{B}^N_{s_0}(p) \times [0,l])
    = P(E,U),
    \end{align*}
    contradicting the fact that $E$ is a perimeter minimizer. Hence, $\partial E$ has only one connected component. 
    
    In conclusion, the closed representative of $E$ is a smooth set with connected boundary. Hence, it is connected.
    \end{proof}
\end{lemma}

\begin{lemma}\label{lem:magia}
    Let $(M^n,g)$ be a Riemannian manifold such that $\liminf_{r \to + \infty} \frac{\ssf{Vol}(B_r(p))}{r^2} < + \infty$, then 
    \begin{equation}\label{eq:magia}
        \int_1^\infty \frac{t^2}{\ssf{Vol}(B_t(p))} \, dt =+\infty.
    \end{equation}
\end{lemma}

\begin{proof}
    By assumption, we can find $t_i\uparrow \infty$ such that 
    \begin{equation}
        \frac{\ssf{Vol}(B_{t_i}(p))}{t_i^2} < C, \qquad \text{for every }i\in \mathbb N,
    \end{equation}
    where $C>0$ is a fixed constant. Then, for every $i\in \mathbb N$ and every $s\in [0,1]$, we have
    \begin{equation}
        \ssf{Vol}(B_{t_i-s}(p)) \leq \ssf{Vol}(B_{t_i}(p)) < C t_i^2 = C (t_i^2-s) \, \frac{t_i^2}{t_i^2-s} \leq  2C (t_i^2-s),
    \end{equation}
    where the last inequality holds for $i$ sufficiently large. Thus, the integrand in \eqref{eq:magia} is greater than $\frac{1}{2C}$ on $\cup_i \, [t_i-1,t_i]$. The thesis easily follows.
\end{proof}

The next result combines \cite[Theorem 2.4]{BlowDown1} and  \cite[Remark 2.1]{zhou2024optimalvolumeboundvolume} (see also \cite[Proposition A.1]{NavarroLinear}). We refer to Definition \ref{DBlowdown} for the definition of blow-down of a manifold with non-negative Ricci curvature.

\begin{lemma} \label{L3}
    Let $(M^n,g)$ be a non-compact manifold with $\Ric_M \geq 0$ and such that
    \[
    \liminf_{r \to + \infty} \frac{\ssf{Vol}(B_r(p))}{r} < + \infty.
    \]
    Then, we have
    \[
    \limsup_{r \to + \infty} \frac{\ssf{Vol}(B_r(p))}{r} < + \infty,
    \]
    and the blow-down of $M$ is unique and it is either a line or a half line.
    \end{lemma}

 %We recall that, given a pointed finite dimensional Alexandrov space with non-negative sectional curvature $(X,\sd,p)$, there exists a unique metric cone $(X_\infty,\sd_\infty,p_\infty)$, such that $(X,\sd/t^2)$ converge to $(X,\sd,p)$ in pGH sense as $t \uparrow + \infty$ (see for instance \cite{AntonelliPozzettaAlexandrov}). In the next theorem we study the blow-down procedure which is at the core of our strategy.
In the next proposition we study the blow-down procedure which is at the core of our strategy.
 
\begin{proposition} \label{P:blowdownDim2}
    Let $(M^n,g)$ be a Riemannian manifold with $\Sec_M \geq 0$ with
    \begin{equation}\label{eq:assumptionVol}
        \liminf_{r \to + \infty} \frac{\ssf{Vol}(B_r(p))}{r^2} < + \infty.
    \end{equation}
    Let $E \subset M$ be a smooth perimeter minimzing set with non-compact boundary.
    % with $p \in \partial E$. 
    % If there exists $c_2 \geq c_1>0$ such that
    % \[
    % c_1 r \leq \aH^{n-1}(\partial E \cap B_r(p)) \leq c_2 r
    % \]
    % for every $r$ sufficiently large,
    % then the blow-down of $M$ has Hausdorff dimension at most $2$.
    Consider the metric space $(X,\sd)$ obtained by gluing $(E,g)$ and $(\partial E,g) \times \bb{R}_+$ along their isometric boundaries. Then, $X$ is an Alexandrov space with non-negative sectional curvature, $E \subset X$ is perimeter minimizing, and the blow-down of $X$ is a cone of Hausdorff dimension $2$.
    \begin{proof}
        % From \eqref{E8}, combining with ..., it follows that $M$ has at most quadratic volume growth.
        % Hence, by Theorem \ref{T1}, $\partial E$ is a smooth totally geodesic submanifold. By Theorem \ref{T2}, $E$ is connected. 
        % Consider the metric space $(X,\sd)$ obtained by gluing $(E,g)$ and $(\partial E,g) \times \bb{R}_+$ along their isometric boundaries.
        By Lemma \ref{lem:magia} and Theorem \ref{T1}, $\partial E \subset M$ is totally geodesic, so that with its intrinsic metric it has non-negative sectional curvature. Hence, since $E$ and $\partial E$ are connected by Lemma \ref{Lemma:connectedness}, $(E,g)$ and $(\partial E,g) \times \bb{R}_+$ are Alexandrov spaces with non-negative sectional curvature and isometric boundaries. By \cite{PetruninGluing}, $(X,\sd)$ is an Alexandrov space with non-negative sectional curvature.

        It is easy to check that $E$ is sub-minimizing and super-minimizing, so that by Lemma \ref{L4} it follows that $E$ is perimeter minimizing.
        We divide the remaining part of the proof in steps.\medskip

        \textbf{Step 1}: The blow-down of $(\partial E,g)$ has Hausdorff dimension $1$.\medskip
        
        \noindent By Theorem \ref{T7}, combined with our assumption \eqref{eq:assumptionVol}, we have that
        \begin{equation} \label{E12}
            \liminf_{r \to + \infty} \frac{P(E, B_r(p))}{r} < + \infty.
       \end{equation}
        Since the distance induced by $g$ on $\partial E$ is larger than the one induced on $M$ restricted to $\partial E$, denoting by $B_r^{\partial E}(p)$ balls in $(\partial E,g)$,
       it holds that
       \[
       \aH^{n-1}(B_r^{\partial E}(p)) \leq \aH^{n-1}(\partial E \cap B^M_r(p))=P(E,B_r(p)).
       \]
       Combining this with \eqref{E12}, it holds
       \begin{equation}
           \liminf_{r \to + \infty} \frac{\aH^{n-1}(B_r^{\partial E}(p))}{r}< + \infty.
       \end{equation}
       Since $(\partial E,g)$ is a manifold of non-negative sectional curvature, by Lemma \ref{L3}, the blow-down of $(\partial E,g)$ is one-dimensional.\medskip

        \textbf{Step 2}: The blow-down of $(X,\sd)$ has Hausdorff dimension $2$. \medskip
        
        % Let $r_i \uparrow + \infty$. and the manifolds $(E,g/r_i^2)$. Consider then the manifolds $(\partial E \times \bb{R}_+,(g+dt^2)/r_i^2)$. Finally, consider metric spaces $(X_i,\sd_i)$ obtained by gluing $(E,g/r_i^2)$ and $(\partial E \times \bb{R}_+,(g+dt^2)/r_i^2)$ along their isometric boundaries. It is immediate to check that $(X,\sd/r_i)=(X_i,\sd_i)$.
        Let $r_i \uparrow + \infty$.
        Let $(X_\infty,\sd_\infty,x_\infty)$ be the pGH limit of the sequence $(X,\sd/r_i,x)$, i.e. the blow-down of $(X,\sd)$. 
        We need to show that such limit space has Hausdorff dimension $2$.

        To this aim, let $p \in \partial E$ and let $\bar{p}_i:=(p,10r_i) \in \partial E \times \bb{R}_+ \subset X$. 
        On the balls $B_{r_i}(\bar{p}_i) \subset X$, we consider the distance $\tilde{\sd}$ obtained as the restriction of the distance induced by $g+dt^2$ on $\partial E \times \bb{R}$.
        % Consider now the space $(B^i_1(\bar{p}_i),\sd/r_i)$ obtained by restricting $\sd/r_i$ to $B^i_1(\bar{p}_i)$. 
        % Consider then the space $(B^i_1(\bar{p}_i),\tilde{\sd}_i)$ obtained by taking the distance $\tilde{\sd}_i$ induced by $(g+dt^2)/r_i^2$ on $B^i_5(\bar{p}_i)$ and then restricting it to $B^i_1(\bar{p}_i)$.
        % We claim that the two metric spaces coincide
        We claim that $\sd$ and $\tilde{\sd}$ coincide on $B_{r_i}(\bar{p}_i)$.
        Indeed,  let ${\gamma} \subset X$ be a curve between two points $p_1,p_2 \in B_{r_i}(\bar{p}_i)$ realizing the distance $\sd(p_1,p_2)$. If ${\gamma} \subset \partial E \times \bb{R}_+$, it follows that $\tilde{\sd}(p_1,p_2) \leq \sd(p_1,p_2)$. At the same time, if ${\gamma} \not\subset \partial E \times \bb{R}_+$, then it connects $p_1$ to a point in $\partial E \times \{0\}$. Hence, $\gamma$ has length greater than $9r_i$, contradicting the fact that $\sd(p_1,p_2) \leq 2r_i$. This proves that $\tilde{\sd}(p_1,p_2) \leq \sd(p_1,p_2)$.
        Let now $\tilde{\gamma} \subset \partial E \times \bb{R}$ be a curve between two points $p_1,p_2 \in B_{r_i}(\bar{p}_i)$ realizing the distance $\tilde{\sd}(p_1,p_2)$. Arguing as before, and using $\tilde{\sd}(p_1,\bar{p}_i) \leq \sd(p_1,\bar{p}_i) \leq r_i$, it follows that $\tilde{\gamma} \subset \partial E \times \bb{R}_+$, so that $\sd(p_1,p_2) \leq \tilde{\sd}(p_1,p_2)$.
        This proves our claim that $\sd$ and $\tilde{\sd}$ coincide on $B_{r_i}(\bar{p}_i)$.

        As a consequence, there exists an open ball $B$ in $(X_\infty,\sd_\infty)$ which arises as GH limit of $(B_{r_i}(\bar{p}_i),\tilde{\sd}/r_i)$. Hence, $B$ is isometric to an open ball in the blow-down of $\partial E \times \bb{R}$, which has Hausdorff dimension $2$ by Step 1. Hence, an open set of $X_\infty$ has Hausdorff dimension $2$. Since $X_\infty$ an Alexandrov space, it also has Hausdorff dimension $2$.
        \end{proof}
        \end{proposition}

\begin{lemma}\label{lem:volumegrowthgluing}
Let $(M^n,g)$ be a Riemannian manifold with $\Sec_M \geq 0$.
    Let $E \subset M$ be a smooth perimeter minimzing set with totally geodesic boundary.
    Consider the metric space $(X,\sd)$ obtained by gluing $(E,g)$ and $(\partial E,g) \times \bb{R}_+$ along their isometric boundaries. Let $p \in \partial E$ and $r>0$. Then, 
    \[
    \aH^n(B_r^X(p)) \leq c(n)\ssf{Vol}(B_r^M(p)).
    \]
    \begin{proof}
        Since $E$ is perimeter minimizing in $X$, by Theorem \ref{T7}, for every $r>0$, it holds
        \[
        \aH^n(B_r^X(p)) \leq c(n) \aH^n(B_r^X(p) \cap E). 
        \]
        By definition of $X$, it holds $B_r^X(p) \cap E=B_r^M(p) \cap E$. Hence,
        \[
        \aH^n(B_r^X(p)) \leq c(n) \aH^n(B_r^M(p) \cap E) \leq c(n)\aH^n(B_r^M(p)), 
        \]
        as claimed.
    \end{proof}
    \end{lemma}

In the following two results we study the two possible blow-downs of the glued space $X$ considered in Proposition \ref{P:blowdownDim2}. We recall that, given a metric space $(Z,\sd_Z)$, we denote by $C(Z)$ the metric cone with section $Z$. In the next theorem, we show that, when $R \in (0,1]$, the cone $C(S^1_R)$ is an $\RCD(0,N)$ space if and only if it is equipped with the 2-dimensional Hausdorff measure $\aH^2$ (up to a constant). 

\begin{thm} \label{P1}
    Let $(C(S^1_R),\sd,\m)$ with $R\leq 1$ be an $\RCD(0,N)$ space. Then $\m=c\aH^2$, for some constant $c>0$.
\end{thm}

\begin{proof}
    Consider the map $\varphi: \bb R^2\setminus ([0,+\infty)\times \{0\}) \to C(S^1_R)$ defined in polar coordinates as 
    \begin{equation}
        \varphi(r, \theta) = \big( r, (R \cos (\theta/R), R \sin (\theta/R))\big).
    \end{equation}
    Intuitively, the map $\varphi$ wraps $\bb R^2\setminus ([0,+\infty)\times \{0\})$ around the cone $C(S^1_R)$. Observe that, by definition, $\varphi$ is a local isometry, thus we can define a measure $\tilde \m$ on $\bb R^2\setminus ([0,+\infty)\times \{0\})$ by requiring that locally $\tilde \m= (\varphi^{-1})_\# \m$. Then, for every point $x\in \bb R^2\setminus ([0,+\infty)\times \{0\})$, there exists a closed convex neighborhood $C\ni x$ such that $(C
    ,\sd_{eu},\tilde \m \mres C)$ is an $\RCD(0,N)$ space. 
    
    Now, for every $\delta>0$, consider the set $V_\delta = (-\infty,-\delta] \times \bb R$. Using the local-to-global property of $\RCD(0,N)$ (see \cite{CavallettiMilman}), we deduce that $(V_\delta,\sd_{eu},\tilde \m \mres V_\delta)$ is an $\RCD(0,N)$ space. Gigli's Splitting Theorem \cite{gigli2013splittingtheoremnonsmoothcontext} ensures that $\tilde \m \mres V_\delta= \mathfrak{n}^\delta\times\lambda^1$ for some measure $\mathfrak{n}^\delta$ on $(-\infty,-\delta]$. As this holds for every $\delta>0$, we conclude that $\tilde \m \mres ((-\infty,0)\times \bb R)= \mathfrak{n}\times\lambda^1$ for some measure $\mathfrak{n}$ on $(-\infty,0)$. Reasoning in the same way, we get that $\tilde \m \mres  ( \bb R \times (0,+\infty))= \mathfrak{n}'\times \lambda^1$ for some measure $\mathfrak{n}'$ on $(0,+\infty)$. Similarly, we obtain an analogous splitting of the measure in the lower half plane. Combining everything, we deduce that $\tilde \m = c \lambda^2= c\aH^2$, for some constant $c>0$. Finally, as $\varphi$ is a local isometry and $\m=\varphi_\# \tilde \m$ locally, we conclude that $\m=c\aH^2$. 
\end{proof}

The next proposition is the core technical result of the paper.

\begin{proposition}\label{prop:main}
    Let $(M^n,g,p)$ be a Riemannian manifold with $\Sec_M \geq 0$ and such that
    \begin{equation}\label{eq:assumption}
        \liminf_{r \to + \infty} \frac{\ssf{Vol}(B_r(p))}{r^2}  \in (0,+\infty).
    \end{equation}
   Let $E \subset M$ be a smooth perimeter minimzing set such that $\partial E$ is non-compact.
    % with $p \in \partial E$. 
    % If there exists $c_2 \geq c_1>0$ such that
    % \[
    % c_1 r \leq \aH^{n-1}(\partial E \cap B_r(p)) \leq c_2 r
    % \]
    % for every $r$ sufficiently large,
    % then the blow-down of $M$ has Hausdorff dimension at most $2$.
    Consider the metric space $(X,\sd)$ obtained by gluing $(E,g)$ and $(\partial E,g) \times \bb{R}_+$ along their isometric boundaries. If the blow-down of $X$ is a cone of the type $C([0,l])$, then $l=\pi$. Moreover, when taking the blow-down, the set $E \subset X$ converges in $\sL^1_{loc}$ sense to $C([0,\pi/2]) \subset C([0,\pi])$ or to $C([\pi/2,\pi]) \subset C([0,\pi])$.
    \end{proposition}
    
    \begin{proof}
        Let $(C([0,l]),\sd_\infty,p_\infty)$ be the blow-down of $(X,\sd,p)$, with $p_\infty$ being the tip of $C([0,l])$. We remark that this blow-down does not depend on the sequence of rescalings. According to  assumption \eqref{eq:assumption} combined with Lemma \ref{lem:volumegrowthgluing}, we consider a sequence $r_i\uparrow +\infty$ such that
        \begin{equation} \label{eq:CrescitascaleRi}
            \aH^n(B_{r_i}(p)) \leq C r_i^2, \qquad \text{for every }i\in \bb N.
        \end{equation}
        
        \textbf{Step 1:} Fix any $t>0$ and consider the sequence of pointed metric measure spaces 
        \begin{equation}\label{eq:sequence}
            \Big(X, \frac{\sd}{r_i/t}, \frac{\aH^n}{\aH^n(B_{r_i/t}(p))},p \Big) \to (C([0,l]),\sd_\infty,\m^t_\infty, p_\infty), 
        \end{equation}
        in pmGH sense (up to a subsequence ), for some blow-down measure $\m^t_\infty$. Then, there exist $0 \leq a_t < b_t \leq l$ such that $Y^t_\infty=C([a_t,b_t]) \subset C([0,l])$ is a perimeter minimizing set in $(C([0,l]),\sd_\infty,\m^t_\infty)$ with
        \begin{equation} \label{Eq:PerimeterLinear}
            P_{\m_\infty^t}(Y_\infty^t, B_t(p_\infty))\leq \tilde C t,
        \end{equation}
        for a constant $\tilde C$ not depending on $t$.
        \medskip
        
        \noindent Consider the closed manifold $(E,g)$ and observe that its induced metric coincides with the restriction of $\sd$ to $E$. Hence, the blow-down of $(E,g)$ is isometric to a subset of the blow-down of $(X,\sd)$. In particular, there is a closed subset $Y^t_\infty \subset C([0,l])$ with $p_\infty \in Y^t_\infty$ such that $(Y^t_\infty,\sd_\infty)$ is isometric to the blow-down of $(E,g)$. 
        Since $(E,g)$ is an Alexandrov space with non-negative sectional curvature, its blow-down is a cone (see for instance \cite[Theorem 2.11]{AntonelliPozzettaAlexandrov}). Moreover, when identifying the blow-down of $(E,g)$ with $Y^t_\infty$, the tip of such cone is identified with the tip $p_\infty \in C([0,l])$.
        Hence, there exist $0 \leq a_t \leq b_t \leq l$ such that $Y^t_\infty=C([a_t,b_t]) \subset C([0,l])$.

        To prove that $Y^t_\infty$ is perimeter minimizing, we are going to show that it is the $\sL^1_{loc}$ limit of $E$, along the sequence in \eqref{eq:sequence}.
        To this aim, call $E^t_\infty$ the perimeter minimizing set in $(C([0,l]),\sd_\infty,\m^t_\infty)$ that arises as $\sL^1_{loc}$ limit of $E$. This set is non-trivial, i.e. its boundary is non-empty, by Proposition \ref{P26}. We consider the closed representative of $E^t_\infty$ and we claim that $Y^t_\infty=E^t_\infty$. 

        In the space $(Z,\sd_Z)$ realizing the pmGH convergence to the blow-down, denoting $\m_i^t=[\aH^n(B_{r_i/t}(p))]^{-1} \aH^n$, the measures $1_E\m^t_i$ converge weakly to the measure $1_{E^t_\infty}\m^t_\infty$ and the set $E$ converges in Hausdorff sense on compact sets of $Z$ to $Y^t_\infty$. 
        We first show that $E^t_\infty \subset Y^t_\infty$. 
        Since (the closed representative of) $E^t_\infty$ is the closure of its open representative, any point $x \in E^t_\infty$ is in the support of $1_{E^t_\infty}\m^t_\infty$. Since $1_E\m^t_i \to 1_{E^t_\infty}\m^t_\infty$ weakly, there exists a sequence of points $x_i \in E$ such that $x_i \to x$ in $(Z,\sd_Z)$. Hence, we deduce $x \in Y^t_\infty$, proving that $E^t_\infty \subset Y^t_\infty$.

        To show the other inclusion, we fix any $x \in Y^t_\infty$. Then, there exists a sequence $x_i \in \bar{E}$ converging to $x$ in $(Z,\sd_Z)$. By Theorem \ref{T7}, $x$ belongs to the support of the limit measure $1_{E^t_\infty} \m^t_\infty$. Hence, $x$ belongs to the closed representative of $E^t_\infty$, proving that $Y^t_\infty=E^t_\infty$. Since $E^t_\infty$ has an open representative, it follows that $a_t<b_t$.

        We proceed now to prove \eqref{Eq:PerimeterLinear}. 
        We use notations $B$, $B^i$, and $B^\infty$ for balls w.r.t. the distances $\sd$, $\sd/(r_i/t)$ and $\sd_\infty$ in the respective spaces.
        By the lower semicontinuity of perimeters under $\sL^1_{loc}$ convergence, we get that 
        \begin{equation} \label{Eq:LscPert}
        P_{\m^t_\infty}(Y^t_\infty,B^\infty_t(p_\infty)) \leq \liminf_{i \to + \infty} P_{\m_i^t}(E,B^i_t(p))
        = \liminf_{i \to + \infty} \frac{(r_i/t) P(E,B_{r_i}(p))}{\aH^n(B_{r_i/t}(p))}.
        \end{equation}        
        % which gives
        % \begin{equation} \label{E9}
        %     \int_0^t f_{|C(\{a\})}(s) \, ds \leq \liminf_{i \to + \infty} \frac{(r_i/t) P(E,B_{r_i}(p))}{\aH^n(B_{r_i/t}(p))}.
        % \end{equation}
        We now denote by $B^{\partial E}$ balls in $\partial E$ w.r.t. the metric induced by $(\partial E,g)$. The manifold $(\partial E,g)$ has non-negative sectional curvature according to Lemma \ref{lem:magia} and Theorem \ref{T1}.  Since $\partial E$ is non-compact, using  \cite{Y}, it holds that
        \begin{equation*}
             \aH^{n-1}(B^{\partial E}_r(p))/r \geq c_1 
        \end{equation*}
        for every $r>0$ sufficiently large and some $c_1 >0$ depending on $\partial E$. 
        Using that $B^{\partial E}_r(p) \subset B_r(p) \cap \partial E$, it follows that for $r>0$ sufficiently large it holds
        \begin{align*}
            c_1 r & \leq \aH^{n-1}(B^{\partial E}_r(p)) \leq  
            \aH^{n-1}(\partial E \cap B_r(p))
             \leq 
            c_2 \aH^{n}(B_r(p))/r ,
        \end{align*}
        for some constant $c_2 >0$ depending only on $n$, where the last step follows from Theorem \ref{T7}. 
        Combining this with \eqref{Eq:LscPert}, \eqref{eq:CrescitascaleRi} and Theorem \ref{T7}, it holds
        \[
        P_{\m^t_\infty}(Y^t_\infty,B^\infty_t(p_\infty)) \leq \tilde C t,
        \]
        for a constant $\tilde C$ not depending on $t$.\medskip
        
       \textbf{Step 2:} The estimate \eqref{Eq:PerimeterLinear} can be improved to 
        \begin{equation} \label{Eq:PerimeterLinear2}
            P_{\m_\infty^t}(Y_\infty^t, B^\infty_s(p_\infty))\leq 20 \tilde C s, \qquad \text{for every }s\in (0,t/2).
        \end{equation}\vspace{-5pt}\\
        Since $(C([0,l]),\sd_\infty,\m^t_\infty)$ is an $\RCD(0,n)$ space, by the stratification results \cite{MondinoNaber,KellMondino,GPstrat,DePhMR,BScon}, it holds $\m^t_\infty=f^t \lambda^2$. By Lemma \ref{L7}, $f^t$ is locally Lipschitz in the interior of its domain. Assume $a_t\neq 0$.
        By \cite{MondinoInventiones}, for $\aH^1$-a.e. line $r$ parallel to $C(\{a_t\})$, the restricted function $f^t_r:r \cap C([0,l]) \to \bb{R}_+$ is a $\ssf{CD}(0,n)$ density on $r \cap C([0,l])$. Hence, $(f_r^t)^{1/(n-1)}$ is concave on $r \cap C([0,l])$. Since $(f_r^t)^{1/(n-1)}$ is positive and $r \cap C([0,l])$ is a half line, $(f_r^t)^{1/(n-1)}$ increases as one moves away from the endpoint. Since $f^t$ is continuous, we conclude that $f^t_r$ is increasing for \emph{every} line $r$ parallel to $C(\{a_t\})$. The same holds for every line parallel to $C(\{b_t\})$, if $b_t\neq l$.

        We now show that, when $a_t\neq 0$, it holds
        \begin{equation} \label{Eq:IntermediateStep2}
        f^t_{|C(\{a_t\})}(z) \leq 10 \tilde{C}, \qquad \text{for every } z \in (0,t/2).
        \end{equation}
        %We denote by $P_\infty$ perimeters in $(C([0,l]),\sd_\infty,\m_\infty)$, by $P_i$ perimeters in $(X,\sd/r_i,\m_i)$, and by $P$ perimeters in $(X,\sd,\aH^n)$.
        By Lemma \ref{L6}, it follows that 
        \[
        \int_0^t f^t_{|C(\{a_t\})}(z) \, dz
        \leq P_{\m_\infty^t}(Y^t_\infty,B^\infty_t(p_\infty)),
        \]
        where in the r.h.s. we see $f^t_{|C(\{a_t\})}$ as a function defined on $\bb{R}_+$. Combining this with Step 1, we deduce  
        \begin{equation*} 
            \int_0^t f^t_{|C(\{a_t\})}(z) \, dz \leq \tilde C t.
        \end{equation*}
         Since $f^t_{|C(\{a\})}$ is increasing, the previous inequality implies \eqref{Eq:IntermediateStep2}.
         
         By an analogous argument, if $b_t \neq l$, \eqref{Eq:IntermediateStep2} holds with $b_t$ in place of $a_t$. Combining Lemma \ref{L6} with \eqref{Eq:IntermediateStep2}, we conclude the proof of Step 2. \medskip

        \textbf{Step 3:} There exist a measure $\m_\infty$ on $C([0,l])$ and $Y_\infty=C([a,b]) \subset C([0,l])$ with $0 \leq a < b \leq l$ such that $Y_\infty$ is a perimeter minimizing set in $(C([0,l]),\sd_\infty,\m_\infty)$ and 
        \begin{equation}
            P_{\m_\infty}(Y_\infty, B^\infty_t(p_\infty))\leq 20 \tilde C t, \qquad \text{for every }t>0.
        \end{equation}\vspace{-5pt}\\
        Let $t_j \uparrow +\infty$. Consider the corresponding $\m_\infty^{t_j}$ and $Y_\infty^{t_j}$, obtained in the previous steps. Since each $\m_\infty^{t_j}$ is a limit of normalized measures, it holds $\m_\infty^{t_j}(B_1^\infty(p_\infty))=1$. Hence, up to a subsequence, $\m_\infty^{t_j} \to \m_\infty$ weakly for some limit measure $\m_\infty$. Similarly, up to a subsequence, $Y_\infty^{t_j} \to Y_\infty$ in $\sL^1_{loc}$ sense for some non-trivial perimeter minimizing set $Y_\infty=C([a,b])$ with $a < b$. By lower semicontinuity of perimeters under $\sL^1_{loc}$ convergence and Step 2, we deduce
        \[
        P_{\m_\infty}(Y_\infty,B^\infty_s(p_\infty))
        \leq 
        \liminf_{j \to +\infty} P_{\m_\infty^{t_j}}(Y_\infty^{t_j}, B^\infty_s(p_\infty))\leq 20 \tilde{C} s \qquad \text{ for every } s>0.
        \]\vspace{-5pt}\\
        \indent \textbf{Step 4:} $l=\pi$ and either $Y_\infty=C([0,\pi/2])$ or $Y_\infty=C([\pi/2,\pi])$.\medskip

        \noindent As before, it holds $\m_\infty=f\lambda^2$, for some function $f$ which is locally Lipschitz in the interior of $C([0,l])$ (Lemma \ref{L7}). Without loss of generality, we assume $a\neq 0$. In $C([0,l]) \setminus C(\{a\})$, consider the distance $\sd_a$ from from $C(\{a\})$. By \cite{Weak} and \cite{lapb}, the function $\sd_a$ is superharmonic in $C([0,l]) \setminus C(\{a\})$ w.r.t. the weighted measure $\m_\infty= f\lambda^2$, which implies that 
        \begin{equation} \label{E10}
        \nabla \sd_a \cdot \nabla f \leq 0.
        \end{equation}
        Let $\pi:C([0,l]) \to C(\{a\})$ be the nearest point projection on $C(\{a\})$ and let $\tilde{f}(x):=f(\pi(x))$. 
        
        As before, for $\aH^1$-a.e. line $r$ parallel to $C(\{a\})$, the restricted function $f_r:r \cap C([0,l]) \to \bb{R}_+$ is a $\ssf{CD}(0,n)$ density on $r \cap C([0,l])$. Since $f$ is continuous, $f_{|C(\{a\})}$ is a $\CD(0,n)$ density as well. In particular, $f_{|C(\{a\})}$ is increasing.
        Then, the space $(C([0,l]),\sd_\infty,\tilde{f}\lambda^2)$ is an $\RCD(0,n+1)$ space, being a convex subset of $C(\{a\})\times \bb{R}$, equipped with the product distance and the product measure $(f_{|C(\{a\})} \lambda^1 ) \times \lambda^1$. 
        %The condition of being a super-minimizer is stable w.r.t. $\sL^1_{loc}$ convergence along sequences of spaces converging in pmGH sense (the same proof of \cite[Proposition 3.9]{ABS19} works CONTROLLA).

        Consider the pmGH limit of the sequence $(C([0,l]),\sd_\infty/i,\tilde{\m}_\infty/\tilde{\m}_\infty(B^\infty_i(p_\infty)))$ as $i \uparrow + \infty$, where we set $\tilde{\m}_\infty:=\tilde{f}\lambda^2$. We claim that such pmGH limit is isomorphic to $(C([0,l]),\sd_\infty,c\lambda^2)$ for some constant $c>0$. 
        % If this is the case, given that the set $C([a,l])$ still minimizes the perimeter w.r.t. outher competitor, we deduce that $a=\pi/2$.
        Indeed, each space $(C([0,l]),\sd_\infty/i,\tilde{\m}_\infty/\tilde{\m}_\infty(B^\infty _i(p_\infty)))$ is isomorphic as a metric measure space to $(C([0,l]),\sd_\infty,\tilde{f}_i \lambda^2)$,
        where
        \begin{equation} \label{E11}
        \tilde{f}_i(x):=\frac{\tilde{f}(ix)}{\int_{B_1^\infty(p_\infty)} \tilde{f}(iz) \, d\lambda^2(z)}.
        \end{equation}
        % with $c_i >0$ being a constant such that
        % \[
        % \int_{B_1(p)} \tilde{f}_i \, d\lambda^2=1.
        % \]
        Since $f_{|C(\{a\})}$ is increasing, it follows by Step 3 and Lemma \ref{L6} that $f_{|C(\{a\})}$ is bounded. Therefore, using the definition of $\tilde{f}$, it holds
        \[
        \lim_{x \to + \infty}\tilde{f}(x)=c'>0.
        \]
        Combining this with \eqref{E11}, we deduce that $\tilde{f}_i \lambda^2$ converges weakly in $C([0,l])$ to $c \lambda^2$ for some $c>0$. This proves that
        \[
        (C([0,l]),\sd_\infty/i,\tilde{\m}_\infty/\tilde{\m}_\infty(B^\infty _i(p_\infty))) \to
        (C([0,l]),\sd_\infty,c\lambda^2)
        \]
in pmGH sense as claimed.

We now claim that $C([a,l])$ 
 is a super-minimizer in $(C([0,l]),\sd_\infty,c\lambda^2)$.
Since $C([a,b])$ is perimeter minimizing in $(C([0,l]),\sd_\infty,f\lambda^2)$, it follows that $C([a,l])$ is a super-minimizer in $(C([0,l]),\sd_\infty,f\lambda^2)$.
        We now denote by $P$ perimeters in $(C([0,l]),\sd_\infty,f\lambda^2)$ and by $\tilde{P}$ perimeters in $(C([0,l]),\sd_\infty,\tilde{f}\lambda^2)$. Let $U \subset C([0,l])$ be a bounded open set and let $C \subset C([0,l])$ be such that $C \supset C([a,l])$ and $C([a,l]) \Delta C \ssubset U$. Using that $f \leq \tilde{f}$ thanks to \eqref{E10}, and that $\tilde{f}=f$ on $C(\{a\})$, using Lemma \ref{L6}, it holds that
        \[
        \tilde{P}(C([a,l]),U)=P(C([a,l]),U) \leq P(C,U) \leq \tilde{P}(C,U).
        \]
        Hence, $C([a,l])$ is a super-minimizer in $(C([0,l]),\sd_\infty/i,\tilde{\m}_\infty/\tilde{\m}_\infty(B^\infty_i(p_\infty)))$ for every $i$, which implies that it is also a super-minimizer in $(C([0,l]),\sd_\infty,c\lambda^2)$. This implies $a \geq \pi/2$.

        If $b<l$, the same argument that we used to show $a \geq \pi/2$, shows that $l-b \geq \pi/2$, so that $l>\pi$, a contradiction. If $b=l$, using again the same argument, it holds $b-a \geq \pi/2$, so that $l=\pi$ and, up to passing to the complement, $Y_\infty=C([\pi/2,\pi])$.
    \end{proof}

\begin{thm}
    Let $(M^n,g,p)$ be a pointed Riemannian manifold with $\Sec_M \geq 0$ and such that
    \[
    \liminf_{r \to + \infty} \frac{\ssf{Vol}(B_r(p))}{r^2} < + \infty.
    \]
    If $E \subset M$ is perimeter minimizing, then $M \cong N \times \bb{R}$ and, with this identification, $E \cong N \times [0,+\infty)$.
    \begin{proof}
    If $\partial E$ is compact, then $(\bar{E},g)$ is a manifold with boundary which is isometric to $\partial E \times \bb{R}_+$ by \cite{Kasue} (using that $E$ is non-compact as it is perimeter minimizing). Applying the same result to the complement of $E$, the statement follows.

    Hence, we assume that $\partial E$ is non-compact.
        Combining Theorem \ref{T1} and Lemma \ref{Lemma:connectedness} with Lemma \ref{lem:magia}, we deduce that $E \subset M$ is a smooth connected set with connected boundary. Moreover, $\partial E$ has non-negative sectional curvature. As done above, consider the metric space $(X,\sd)$ obtained by gluing $(E,g)$ and $(\partial E,g) \times \bb{R}_+$ along their isometric boundaries. 
        % We claim that $X \equiv \partial E \times \bb{R}$ and that with this identification $E \equiv \partial (-\infty ,0)$
        % We now prove the claim. 
        Let $(X_\infty,\sd_\infty)$ be the blow-down of $(X,\sd)$. 
               
        By Proposition \ref{P:blowdownDim2}, $X_\infty$ is a cone of Hausdorff dimension $2$.
        Hence, either $X_\infty=C([0,l])$ for some $l\in(0,\pi]$ or $X_\infty=C(S^1_R)$ for some $R\in(0,1]$.
         Moreover, there exists a measure $\m_\infty$ on $X_\infty$ such that $(X_\infty,\sd_\infty,\m_\infty)$ is an $\RCD(0,n)$ space and there exists a non-trivial perimeter minimizing (w.r.t. $\m_\infty$) set $E_\infty \subset X_\infty$.

         \medskip
         \textbf{Step 1}: $X_\infty \cong Y \times \bb{R}$ for some metric space $(Y,\sd_y)$ and with this identification $E_\infty \cong Y \times \bb{R}_+$.\medskip
        
        \textbf{Case 1}: $X_\infty=C([0,l])$ for some $l\in(0,\pi]$.\medskip
        
        \noindent If it holds that 
        \[
        \liminf_{r \to + \infty} \frac{\ssf{Vol}(B_r(p))}{r^2}  =0,
        \]
        then by Theorem \ref{T7} it follows
        \[
        \liminf_{r \to + \infty} \frac{\ssf{Vol}_{\partial E}(B^{\partial E}_r(p))}{r} \leq \liminf_{r \to + \infty} \frac{\aH^{n-1}(B_r(p) \cap \partial E)}{r}  =0,
        \]
        so that $\partial E$ is compact by \cite{Y}.
        Otherwise, Proposition \ref{prop:main} guarantees that $X_\infty \cong Y \times \bb{R}$ for some metric space $(Y,\sd_y)$ and that $E_\infty \cong Y \times \bb{R}_+$.\medskip

        \textbf{Case 2}: $X_\infty=C(S^1_R)$ for some $R\in(0,1]$. \medskip
        
        \noindent By Theorem \ref{P1}, it holds $\m_\infty=c\aH^2$. By \cite[Proposition 6.26]{Weak}, it holds $R=1$, so that $\X_\infty =\bb{R}^2$. Since the only perimeter minimizing set in $\bb{R}^2$ is the half space, also this case is concluded.
        \medskip
        
        \textbf{Step 2} $M \cong N \times \bb{R}$ and $E \cong N \times [0,+\infty)$. \medskip

        To prove the claim, we will show that $(\bar{E},g)$ is isometric to $\partial E \times [0,+\infty)$. If this is the case, by a mirrored argument, it holds that $(\overline{M \setminus E},g)$ is isometric to $\partial E \times [0,+\infty)$ as well, so that the claim follows.

        Let $r \subset X$ be a ray starting from $\partial E$, contained in $\partial E \times \bb{R}_+$, and such that
        \begin{equation}\label{eq:mindistE}
            \sd(E,r(t))=t \quad \text{for every } t>0.
        \end{equation}
        When considering the rescaled spaces $(X,\sd/i)$ converging to the blow-down $X_\infty$, since $E$ converges to $E_\infty \subset X_\infty$ (both in $\sL^1_{loc}$ sense and in Hausdorff sense in the space realizing the pGH convergence), the ray $r$ converges to a ray $r_\infty \subset X_\infty$ with the property that
        \[
        \sd_\infty(E_\infty
        ,r_\infty(t))=t \quad \text{for every } t>0.
        \]
        By Step 1, it holds $X_\infty \cong Y \times \bb{R}$ and, with this identification, $E_\infty \cong Y \times \bb{R}_+$. By reflecting, we obtain a new identification
        $X_\infty \cong Y \times \bb{R}$ such that $E_\infty \cong Y \times \bb{R}_-$.
        The ray $r_\infty$ is now one half of a line $\gamma_\infty=\{y_\infty\} \times \bb{R}$ for some $y_\infty \in Y$. We parametrize $\gamma_\infty$ so that $\gamma_\infty(t)=(y_\infty,t)$.

        % We now adapt an argument from \cite{MR4393128}. 
        We use the line $\gamma_\infty \subset X_\infty$ to construct a line $\gamma \subset X$ containing the half line $r$.
        To this aim, consider the points $p_{-1,\infty}=\gamma_\infty(-1)$, $p_{0,\infty}=\gamma_\infty(0)$, and $p_{1,\infty}=\gamma_\infty(1)$. It holds $p_{1,\infty}=r_\infty(1)$, and $p_{0,\infty}=r_\infty(0)$. Consider points $p_{-1,i} \in (X,\sd/i)$ such that $p_{-1,i} \to p_{-1,\infty}$ in the space realizing the pGH convergence. We then set $p_{0,i}:=r(0)$ and $p_{1,i}=r(i)$. Since $r$ converges to $r_\infty$, it holds $p_{0,i} \to p_{0,\infty}$ and $p_{1,i} \to p_{1,\infty}$ in the space realizing the pGH convergence.

        Consider a length minimizing geodesic $\tilde{\gamma}_i:[-i,0] \to X$ parametrized by constant speed joining $p_{-1,i}$ to $p_{0,i}$. We remark that the speed might be different from $1$. Let $\gamma_i:[-i,i] \to X$ be the curve obtained by gluing $\tilde{\gamma}_i$ and $r_{|[0,i]}$. 
        % Observe that as minimizing geodesic from $p_{0,i}$ to $p_{1,i}$, we can consider the restriction of the ray $r$.
        %By the argument of \cite{MR4393128}\todo{expand}, the curves $\gamma_i$ converge to a line $\gamma$. 
        We follow an argument from \cite{MR4393128} to prove that the curves $\gamma_i$ converge to a line $\gamma$. Given any $\varepsilon>0$, since $p_{j,i}\to p_{j,\infty}$ as $i\to\infty$ for $j=-1,0,1$, we have that, for $i$ sufficiently large,
        \begin{equation}
            \sd(p_{-1,i}, p_{0,i}) \leq (1+\varepsilon) i, \quad \sd(p_{0,i}, p_{1,i}) = i, \quad \text{and} \quad \sd(p_{-1,i}, p_{1,i}) \geq (1-\varepsilon) 2 i.
        \end{equation}
        Now take any $s\geq 0$. From the triangle comparison, we deduce that, for $i$ large enough, it holds
        \begin{equation} \label{E:SplitSec1}
            \sd(\gamma_i(-s), \gamma_i(s)) \geq \frac si \, \sd(p_{-1,i}, p_{1,i}) \geq 2 s(1-\varepsilon).
        \end{equation}
        On the other hand, we have that
        \begin{equation} \label{E:SplitSec2}
            \mathrm{length} (\gamma_{i|[-s,s]}) = \frac si [\sd(p_{-1,i}, p_{0,i}) + \sd(p_{0,i}, p_{1,i})] \leq (2+\varepsilon) s.
        \end{equation}
        Let $\gamma \subset X$ be the curve arising as limit of the curves $\gamma_i$ (modulo a subsequence). For every $s \geq 0$, combining \eqref{E:SplitSec1} and \eqref{E:SplitSec2}, it holds
        \[
        \sd(\gamma(-s),\gamma(s)) \geq 2s, \quad \text{ and } \quad \mathrm{length}(\gamma_{|[-s,s]}) \leq 2s.
        \]
        Hence, $\gamma$ is a line and, by construction, $\gamma$ contains the half line $r$ as claimed. 
        By the Splitting Theorem, there exists an isometry $\phi:X \to N \times \bb{R}$ such that $\phi(\gamma)=\{n\} \times \bb{R}$ and $\phi(r)=\{n\} \times [0,+\infty)$ for some $n \in N$.

        We now conclude the proof of Step 2.  Observe that, since $r$ satisfies \eqref{eq:mindistE} and $X \setminus \bar{E} = \partial E \times (0,+\infty)$, we have 
        \[
        X \setminus \bar{E}=\bigcup_{R>0}B^X_R(r(R)).
        \]
        Hence, using that $\phi(r)=\{n\} \times [0,+\infty)$ for some $n \in N$, it holds
        \[
        \phi(X \setminus \bar{E})=\bigcup_{R>0}\phi(B^X_R(r(R)))=
        \bigcup_{R>0}B^{N \times \bb{R}}_R((n,R))=N \times (0,+\infty).
        \]
        We deduce that $\phi(\bar{E})=N \times (-\infty,0]$, therefore $N\cong \partial E$ and $\bar E \cong\partial E \times [0,+\infty)$ as desired.
    \end{proof}
\end{thm}

\sloppy
\printbibliography
\end{document}